\documentclass{amsart}
\usepackage{amssymb}
\usepackage{amsmath}

\usepackage{bbm}
\usepackage{mathrsfs}
\usepackage[normalem]{ulem}

\newtheorem{lem}{Lemma}[section]

\newtheorem{dfn}[lem]{Definition}
\newtheorem{pro}[lem]{Proposition}
\newtheorem{thm}[lem]{Theorem}
\newtheorem{rem}[lem]{Remark}

\newtheorem{cor}[lem]{Corollary}

\numberwithin{equation}{section}

\usepackage{colortbl}

\def\F{\mathcal F}

\def\N{\mathbb N}

\def\C{\mathcal C}
\newcommand\hdim{\dim_{\mathrm H}}

\def\R{\mathfrak R}
\def\Z{\mathbb Z}

\def\W{\mathcal W}
\def\q{\bold q}

\def\t{\bold t}
\def\w{\bold w}
\def\a{\bold a}
\def\K{\mathcal K}

\newcommand\da{Diophantine approximation}

\newcommand{\ignore}[1]{}

\newcommand {\new}[1]   {\textcolor{blue}{#1}}

\newif\ifdraft\drafttrue

\draftfalse

\newcommand\eq[2]{{\ifdraft{\ \tt [#1]}\else\ignorespaces\fi}\begin{equation}\label{eq:#1}{#2}\end{equation}}
\newcommand {\equ}[1]     {\eqref{eq:#1}}

\title[Measure of limsup sets defined by rectangles]
 {Measure theoretic laws for limsup sets \\ defined by rectangles} 

\author[D. Kleinbock]{Dmitry ~Kleinbock}
\address {Dmitry Kleinbock, Brandeis University, Waltham MA 02454-9110.} \email{kleinboc@brandeis.edu}
\author[B. Wang]{Baowei Wang}
\address{Bao-wei Wang, School  of  Mathematics  and  Statistics,  Huazhong  University  of Science  and  Technology, 430074 Wuhan, China}
 \email{bwei\_wang@hust.edu.cn}

\thanks{DK was supported by NSF grants DMS-1900560 and DMS-2155111. This material is based upon work supported by a grant from the Institute for Advanced Study School
of Mathematics.}


\begin{document}

\maketitle

\begin{abstract}
In this paper, we present a general principle for the Lebesgue measure theory of limsup sets defined by rectangles under the hypothesis of {\em ubiquity for rectangles}.
\end{abstract}

\section{Introduction}

Metric Diophantine approximation is originated from the study of how well a real number can
be approximated by rationals in the sense of measure and dimension. Since the pioneer work of Khintchine \cite{Khint}, Jarn\'{i}k \cite{Jarn1}, Groshev \cite{Gro}, Sprindzuk \cite{Sprindzuk} and Schmidt \cite{Sch1} where the metric theory of concrete sets in classic Diophantine approximation are presented, a corner stone is the introduction of the terminology of {\em regular systems} by Baker \& Schmidt \cite{BakS} in 1970, which has opened up a possibility about general principles for the metric theory of limsup sets.
\begin{dfn}[Regular system]
 Let $I$ be an interval in $\mathbb{R}$ and let $\Gamma=\{\gamma_n\}_{n\ge 1}$ be a sequence of real numbers in $I$, together with a positive function $\mathcal{N}:\Gamma\to \mathbb{R}_+$. Call $(\Gamma, \mathcal{N})$ a regular system, if for any subinterval $J$ of length $|J|$  there exists an integer $K_J$ such that for any $K\ge K_J$, there exist elements $\gamma_{1},\dots, \gamma_{t}$ in $\Gamma\cap J$ such that \begin{equation*}\label{0}
t\ge c(\Gamma)\cdot |J|\cdot K,\ \ {\text{and}}\ \ \mathcal{N}(\gamma_i)\le K,\ |\gamma_{i}-\gamma_{j}|\ge K^{-1}  \ {\text{for all}}\ 1\le i\ne j\le t,
\end{equation*} where $c(\Gamma)$ is an absolute constant.
\end{dfn}

Equipped with the assumption of $(\Gamma, \mathcal{N})$ being a regular system, one is able to set up a complete metric theory for sets of the form$$
\{x\in I: |x-\gamma|<\psi\big(\mathcal{N}(\gamma)\big), \ {\text{for infinitely many}}\ \gamma\in \Gamma\}.
$$ See Baker \& Schmidt \cite{BakS} for Hausdorff theory and Beresnevich \cite{Ber99} for Lebesgue theory among others. As special cases, it can be applied to the approximation by rational numbers and algebraic numbers.

 The notion of  regular systems in $\mathbb{R}$ was extended to 
 ubiquitous systems in $\mathbb{R}^d$ by Dodson, Rynne \& Vickers \cite{DoRV}
in 1990 to establish general principle for the Hausdorff theory of limsup sets in higher dimensional spaces. This  includes \da\ for systems of linear forms and beyond. In 2006  the notion of ubiquitous systems in $\mathbb{R}^d$ was further generalized to the setting of abstract metric spaces by Beresnevich, Dickinson \& Velani \cite{BDV06}. 
In this paper we call it {\em ubiquity for balls}.
\begin{dfn}[Ubiquity for balls]\label{def_ub}
  Let $X$ be a locally compact metric space with a finite Borel measure $\mu$. Let $\{\R_{\alpha}: \alpha\in J\}$ be a sequence of subsets in $X$ and let $\beta: J\to \mathbb{R}_+$ be a positive function attaching a weight to $\alpha\in J$. Let $\rho: \mathbb{R}_+\to \mathbb{R}_+$ be a function (henceforth referred to as the ubiquitous function), and let  $\{\ell_n, u_n: n\ge 1\}$ be two sequences of positive numbers with $u_n\ge \ell_n\to \infty$ as $n\to \infty$. Call $\big(\{\R_{\alpha}\}_{\alpha\in J}, \rho\big)$ a ubiquitous system  if for any ball $B\subset X$ there exists $n_o(B)$ such that for all $n\ge n_o(B)$\begin{equation}\label{fff7}
  \mu\left(B\cap \bigcup_{\alpha\in J: \ell_n\le \beta_{\alpha}\le u_n}\Delta\big(\R_{\alpha}, \rho(u_n)\big)\right)\ge c(J)\cdot \mu(B),
  \end{equation} where $\Delta(\R_{\alpha},\epsilon)$ denotes the $\epsilon$-neighborhood of $\R_{\alpha}$, and $c(J)$ is an absolute constant.
\end{dfn}


Equipped with the assumption of ubiquity, Beresnevich, Dickinson \& Velani \cite{BDV06} established a complete metric theory on the set of the form $$
\W(\psi):=\ \big\{x\in X: x\in \Delta\big(\R_{\alpha}, \psi(\beta_{\alpha})\big), \ {\text{for infinitely many}}\ \alpha\in J\big\},
$$ including Hausdorff measures as well as $\mu$-measure 
under
some mild and natural conditions. Instead of citing the full generality of the measure theory in \cite{BDV06}, we state the following special case (all the necessary notation will be explained in \S\ref{fw}).

\begin{thm}[Beresnevich, Dickinson \& Velani \cite{BDV06}]\label{t3}
Assume the $\delta$-Ahlfors regularity for a measure $\mu$ on $X$, the $\kappa$-scaling property for every resonant set $\R_{\alpha}$,   $\alpha\in J$, and the ubiquity for balls with respect the function $\rho$ and the sequences $\{\ell_n, u_n: n\ge 1\}$. Assume that $\psi$ is decreasing and that either $\psi$ or $\rho$ is $\lambda$-regular for some $0 < \lambda < 1$. Then $$
\mu\big(\W(\psi)\big)=\mu(X)  \ \ \ {\text{if}}\ \ \ \ \sum_{n\ge 1}\left(\frac{\psi(u_n)}{\rho(u_n)}\right)^{\delta(1-\kappa)}=\infty.
$$
\end{thm}

Among many of applications of Theorem \ref{t3} in \cite{BDV06}, we cite one application to the classical Diophantine approximation, i.e.\ the Khinthchine-Groshev theorem. \begin{cor}\cite{BDV06}
Let $\psi:\N\to \mathbb{R}_+$ be a non-increasing function. Then the set $$
\Big\{A\in [0,1]^{dh}: \|A_i\q\|<\psi(|\q|), \ \forall \ 1\le i\le d, \ {\text{for infinitely many}}\ \q\in \mathbb{Z}^h\Big\}
$$ is of Lebesgue measure zero or one according to $$
\sum_{q=1}^{\infty}q^{h-1}\psi(q)^{d}<\infty\ {\text{or}}\ =\infty.
$$
\end{cor}
\noindent Here $A_i$ are the rows of $A$, $|\q|=\max_{1\le k\le h}|q_k|$, and $\|x\|$ stands for the distance of a real number $x$ from integers.
\smallskip

These general principles have become fundamental 
 in metric number theory.
 Besides the wide usage in classical Diophantine approximation (for example \cite{BBD02,BerD,Bug,Bug03,DickV,Dod92,Dodson,Lev,Ryn92}), they are also used in   Diophantine approximation of $p$-adic fields and formal power fields \cite{Dick99, Kris}, as well as Diophantine approximation on manifolds (for example \cite{BadBV,Ber12,B2,B1}).

It should be noted that all these principles are designed to attack the metric theory of limsup sets defined by balls, that is, starting from
 Dirichlet's theorem in simultaneous \da.
However, more generally  one can 
take Minkowski's theorem,  see \cite{Mi,Sch}, as a starting point.
Here is a special case:

{\em  For any $x=(x_1,\dots,x_d)\in \mathbb{R}^d$, any non-negative vector $(\hat{a}_1,\dots,\hat{a}_d)$ with $\hat{a}_1+\cdots+\hat{a}_d=1$, and any $Q\in \N$, there exists an integer $1\le q\le Q$ such that $$
\|qx_i\|<Q^{-\hat{a}_i}, \ 1\le i\le d.
$$
Consequently, by letting $a_i=1+\hat{a}_i$, $1\le i\le d$, there exist infinitely many integers $p_1,\dots, p_d,q$ such that \begin{equation*}\label{c17}
|x_i-p_i/q|<q^{-a_i}, \ 1\le i\le d.
\end{equation*}
This means that all vectors will fall into a sequence of rectangles centered at rational vectors infinitely often.}

Minkowski's theorem provides a more profound understanding of the distribution of rational vectors,
which works sufficiently well in high dimensional Diophantine approximation compared with Dirichlet's theorem
(for example, Minkowski's theorem intervenes as an essential tool in  Diophantine approximation on manifolds,
see e.g.\ \cite{B1, B2}).
So it should be valuable to study the improvement based on Minkowski's theorem, or more precisely, to consider the metric theory of limsup sets defined
by rectangles. However, besides some specific examples found in the work  of Sprindzuk \cite{Sprindzuk}, Schmidt \cite{Sch1}, Gallagher \cite{Gall}
and Hussain \& Yusupova \cite{HusY} on linear forms, 
no general theory has been put forth. So the study on limsup sets defined by rectangles lags much behind the study on limsup sets defined by balls. In this paper  we hope to push this forward by presenting a general principle for the measure theory of limsup sets defined by rectangles.

\medskip

\noindent{\bf Acknowledgments.}
The first-named author is  grateful to Victor Beresnevich 
and Mumtaz Hussain for useful discussions.


\section{The Framework and Main Results}\label{fr}
In this section, we describe our general framework 
and state the main result of the paper. In fact, one of the major tasks is to find 
suitable assumptions which could possibly catch the nature for the metric theory of limsup sets defined by rectangles.

\subsection{The Framework}\label{fw} Throughout, fix an integer $d\ge 1$. 
Let $(X_i, {\text{dist}}_i, \mu_i)$ be a bounded locally compact metric
space with $\mu_i$ a Borel probability measure and $\text{dist}_i$ a metric on $X_i$ for each $1\le i\le d$.
We consider the product space $(X, {\text{dist}}, \mu)$, where $$
X=\prod_{i=1}^dX_i; \   
\mu=\prod_{i=1}^d \mu_i; \ \ 
{\text{dist}}(x,y)=\max_{1\le i\le d}{\text{dist}}_i(x_i, y_i).
$$
A ball $B(x,r)$ in  $X$ is 
the product of balls in $\{X_i\}_{1\le i\le d}$, i.e. $$
B(x,r)=\prod_{i=1}^d B(x_i,r), \ {\text{for}}\ x=(x_1,\dots, x_d).
$$
\begin{itemize}
\item Let $J$ be an infinite countable index set and let $\beta: J\to \mathbb{R}_+$ be a positive function such that for any $M>1$, $\{\alpha\in J: |\beta(\alpha)|<M\}$ is finite;

\item Let $\{\ell_n, u_n: n\ge 1\}$ be two sequences of integers such that $u_{n}\ge {\ell}_{n}\to \infty$  as $n\to \infty$, and
define $$
J_n=\{\alpha\in J: \ell_{n}\le\beta(\alpha)\le u_{n}\}.
$$

\item Let $\rho=(\rho_1,\dots, \rho_d)$ be a $d$-tuple of functions with $\rho_i:\mathbb{R}_+\to \mathbb{R}_+$ and 
$\rho_i(u)\to 0$ as $u\to \infty$ for each $1\le i\le d$.
\end{itemize}

For each $1\le i\le d$, let $\{\R_{\alpha, i}: \alpha\in J\}$ be a sequence of subsets of $X_i$. The resonant sets in $X$ we are considering are $$
\Big\{\R_{\alpha}=\prod_{i=1}^d\R_{\alpha, i}, \ \ \alpha\in J\Big\}.
$$ 
For any $\bold{r}=(r_1,\dots, r_d)$, denote a rectangle-like set as $$\Delta(\R_{\alpha}, \bold{r})=\prod_{i=1}^d\Delta(\R_{\alpha,i}, r_i),
$$ where $\Delta(\R_{\alpha,i}, r_i)$ is the $ r_i$-neighborhood of $\R_{\alpha,i}$ in $X_i$.

Let $\Psi=(\psi_1,\dots, \psi_d)$ be  a $d$-tuple of positive functions defined on $\mathbb{R}_+$. The set we would like to describe is: $$
\W(\Psi)=\Big\{x\in X: x\in \prod_{i=1}^d\Delta\Big(\R_{\alpha,i}, \psi_i\big(\beta(\alpha)\big)\Big), \ {\text{i.m.}}\ \alpha\in J\Big\},
$$ i.e.\ the set of points which fall into the `{\em rectangle}' $
\prod_{i=1}^d\Delta(\R_{\alpha,i}, \psi_i\big(\beta_{\alpha})\big)$  for infinitely many $\alpha\in J$.

\smallskip

Next we impose some regularity assumptions on the measures $\mu_i$ and the resonant sets $\R_{\alpha}$.
\begin{dfn}[$\delta$-Alhfors regularity] Let $\Omega$ be a complete metric space and $\nu$ be a Borel probability measure. Call $\nu$ 
Ahlfors regular with exponent $\delta$ if there exist constants $c>0$ and $r_0>0$ such that for any $x\in \Omega$ and $r<r_0$,
  $$c^{-1}r^{\delta}\le \nu\big(B(x,r)\big)\le c r^{\delta}.$$
\end{dfn}

In the following, we will assume the measure $\mu_i$ to be $\delta_i$-Ahlfors regular for each $1\le i\le d$. We will also require the resonant sets to have a special form generalizing the Euclidean space set-up where the resonant sets are points or, more generally,  affine subspaces.
\begin{dfn}[$\kappa$-scaling property] Let $0\le \kappa<1$. For each $1\le i\le d$, say that $\{\R_{\alpha, i}\}_{\alpha\in J}$ has a
$\kappa$-scaling property if for any $\alpha\in J$ and any ball $B(x_i,r)$ in $X_i$ with center $x_i\in \R_{\alpha, i}$ and $0<\epsilon<r$,
one has \begin{equation*}
c^{-1}\cdot r^{\delta_i\kappa}\cdot \epsilon^{\delta_i (1-\kappa)}\le \mu_i\big(B(x_i,r)\cap \Delta(\R_{\alpha, i}, \epsilon)\big)\le c \cdot r^{\delta_i\kappa}\cdot\epsilon^{\delta_i (1-\kappa)}
\end{equation*} for some absolute constant $c>0$.\end{dfn}
We list some examples for which the $\kappa$-scaling property holds.\begin{enumerate}
\item  For each $\alpha\in J$, the $i$th coordinate $\R_{\alpha, i}$ is a point in $X_i$, so $\kappa=0$.\smallskip

\item  Let $X_i=\mathbb{R}^n$. For each $\alpha\in J$, the $i$th coordinate $\R_{\alpha, i}$ is an $l$-dimensional affine subspace in $X_i$, so $\delta_i=n$ and $\kappa=l/n$.\smallskip

\item Let $X_i=\mathbb{R}^n$, and  for all $\alpha\in J$ let $\R_{\alpha, i}$ be an $l$-dimensional smooth compact manifold embedded in $X_i$. Then $\delta_i=n$ and $\kappa=l/n$.\smallskip

\item  Let $X_i=\mathbb{R}^n$ and  for all $\alpha\in J$ let $\R_{\alpha, i}$ be a self-similar set   of Hausdorff dimension $l$ satisfying the open set condition. Then $\delta_i=n$ and $\kappa=l/n$.
\end{enumerate}
For a proof of the scaling property in the last two examples, one is referred to Allen \& Baker \cite{AllBa}.
\begin{dfn}[$\lambda$-regularity]
  Let $0<\lambda<1$. A function $\varphi$ is said to be $\lambda$-regular with respect to the sequence $\{u_n\}_{n\ge 1}$ if $$\varphi(u_{n+1})\le \lambda\cdot \varphi(u_n)  \ \ {\text{for all}}\ \ n\gg 1.$$
\end{dfn}

\subsection{Ubiquitous systems for rectangles}

The ubiquity condition for balls (\ref{fff7}) 
is mainly rooted in Dirichlet's theorem in Diophantine approximation \cite{BDV06}. 
Thus, as far as the metric theory
of limsup sets defined by rectangles is concerned, it is reasonable to expect that Minkowski's theorem should intervene in some form.
The following notion of  {\em ``ubiquity for rectangles"} is designed to catch the nature of the rectangles inspired by Minkowski's theorem.
 It first appeared in 
 the previous  work of the second-named author with Xu and Wu \cite{WWX,WangW},
where the Hausdorff theory for limsup sets defined by rectangles was investigated.


\begin{dfn}[Ubiquity for rectangles]\label{dfn1} Call $(\{\R_{\alpha}\}_{\alpha\in J}, \beta)$ a ubiquitous system for rectangles with respect to the function $\rho=(\rho_1,\dots, \rho_d)$ and the sequences $\{\ell_n, u_n: n\ge 1\}$ if there exist a constant $c>0$ such that for any ball $B$ in $X$  \begin{equation}\label{ff5}
\mu\left(B\cap \bigcup_{\alpha\in J_n}\prod_{i=1}^d\Delta\big(\R_{\alpha,i}, \rho_i(u_n)\big)\right)\ge c\cdot \mu(B)  \ \ {\text{for all}}\ \ n\ge n_o(B).
\end{equation}\end{dfn}

Our main result is the following general principle for the measure theory of limsup sets defined by rectangles, which together with the Hausdorff theory developed
 in \cite{WangW} provides a rather complete metric theory for 
this set-up (under the ubiquity hypothesis). For ease of notation, for two $d$-tuple\new{s} of functions $\rho$ and $\Psi$  we write\begin{align*}
 \rho\ {\text{is $\lambda$ regular}}&\Longleftrightarrow\, \rho_i\ {\text{is $\lambda$ regular for all}} \ 1\le i\le d;\\
\Psi(u)\le \rho(u)\ &\Longleftrightarrow \,\psi_i(u)\le \rho_i(u)\ {\text{for all}} \ 1\le i\le d.
\end{align*}
\begin{thm}[Measure theory]\label{t1}
Under the setting given above, assume the $\delta_i$-Ahlfors regularity for $\mu_i$ with $1\le i\le d$, the $\kappa_i$-scaling property for every $\R_{\alpha, i}$ with $\alpha\in J$ and $1\le i\le d$, and the  ubiquity for rectangles. Assume that $\Psi$ is decreasing, that either $\Psi$ or $\rho$ is $\lambda$-regular, and that $\Psi(u_n)\le \rho(u_n)$ for all $n\gg 1$. Then $$
\mu\big(\W(\Psi)\big)=\mu(X) \ \ \ {\text{if}}\ \ \ \ \sum_{n\ge 1}\prod_{i=1}^d\left(\frac{\psi_i(u_n)}{\rho_i(u_n)}\right)^{\delta_i(1-\kappa_i)}=\infty.
$$
\end{thm}

The first application of our result is a solution to a simple shrinking target problem. Though it can be proved by the exponential mixing property of
the underlying dynamics, however with the help of Theorem \ref{t1}
one can see that the proof uses only very basic arithmetic features of the system. \begin{thm}\label{t2}
  Let $b_1,\dots,  b_d\ge 2$ be integers, and let $T_i(x)=b_i x\ (\bmod\ 1)$. Then for any $x_o\in [0,1]^d$ and a $d$-tuple $\Psi$ of positive functions
  $\psi_1,\dots, \psi_d:\mathbb{N}\to \mathbb{R}_+$, the Lebesgue measure of the set $$
  \mathfrak{S}(\Psi)=\Big\{x\in [0,1]^d: 
  |T_i^nx_i-x_{o,i}|<\psi_i(n)  \ \forall \, 1\le i\le d,
  \ {\text{for infinitely many}}\ n\in \N\Big\}
  $$
 is zero or one according to $$
  \sum_{n=1}^{\infty}\prod_{i=1}^d\psi_i(n)<\infty\ {\text{or}}\ =\infty.
  $$
\end{thm}

We next apply our result to 
systems of linear forms mainly to illustrate the way for choosing the ubiquity function. Let $\varphi = \{\varphi_i\}_{1\le i\le d}$ be a $d$-tuple of non-increasing positive functions defined on $\mathbb{Z}_{\ge 0}$ with $$\varphi_i(q)\to 0, \ {\text{as}}\ q\to \infty,$$
 and let $\Phi = \{\Phi_k\}_{1\le k\le h}$ be  an $h$-tuple of non-decreasing positive integer-valued functions defined on $\N$ with
$$\Phi_k(u)\to \infty \ {\text{as}}\ u\to \infty.$$
Consider the following set:
\begin{align*}
 W(\varphi,\Phi):&= \bigg\{A\in [0,1]^{dh}: {\text{the system}} \left\{
                        \begin{array}{ll}
                          \|A_i\q\|<\varphi_i(u), 1\le i\le d, & \\
                          |q_k|\le\Phi_k(u), \ \ \, 1\le k\le h, &
                        \end{array}
                      \right. \\
                     & \qquad \qquad\qquad\qquad\qquad\qquad {\text{has a solution in $\q\in \Z^{h}\smallsetminus \{0\}$}}\ {\text{for infinitely many}}\ u\in \N
\bigg\}\\
&=\big\{A\in [0,1]^{dh}: \|A_i\q\|<\varphi_i\big(\max\{\Phi_1^{-1}(|q_1|^+),\dots, \Phi_h^{-1}(|q_h|^+)\}\big), \ 1\le i\le d, \ {\text{i.m.}}\ \q\in \Z^h\big\}.
\end{align*}
Here $|q|^+$ stands for $\max\{1, |q|\}$. Also in what follows we will denote Lebesgue measure on Euclidean spaces by $\mathcal{L}$.

\begin{thm}\label{schmidt}Assume that there exists $M>1$ such that for all $n\gg 1$, $$c_1\Phi_k(M^{n})\le \Phi_k(M^{n+1})\le c_2 \Phi_k(M^n), \ 1\le k\le h$$ for some absolutely constants $c_1, c_2>1$. Then
$\mathcal{L}\big(W(\varphi, \Phi)\big)$ is zero or one according to
$$
 \sum_{q=1}^{\infty}q^{-1}\cdot \prod_{i=1}^d\varphi_i(q)\cdot \prod_{k=1}^h\Phi_k(q)<\infty \ {\text{or}}\ =\infty.
$$
\end{thm}

 It should be mentioned that Sprind\v zuk \cite{Sprindzuk} established a metric result for systems of linear forms which goes beyond the 
 set-up involving rectangles. Though in Sprind\v zuk's result  only primitive vectors $\q \in \Z^h$ are involved, Theorem \ref{schmidt} can be obtained from %
 \cite[Chapter 1, Theorem 13]{Sprindzuk} 
 with the help of an elementary calculation.

It is instructive to state the special case $h=1$ of the above theorem. Then one can take $\Phi_1(q) = q$ and thus study the set
$${W}(\varphi):=\big\{x\in [0,1]^d: \|qx_i\|<\varphi_i(q), \ \forall\ 1\le i\le d, \ \ {\text{i.m.}}\ q\in \mathbb{N}\big\}.
$$
Theorem \ref{schmidt} immediately implies
\begin{cor} \label{corsim}
 Let $\varphi = \{\varphi_i\}_{1\le i\le d}$ be as 
 above; then
 $\mathcal{L}\big(W(\varphi)\big)$
 is zero or one according to$$
\sum_{q=1}^{\infty}\prod_{i=1}^d\varphi_i(q)<\infty\ {\text{or}}\ =\infty.
$$
\end{cor}

\begin{rem} \rm The
necessity of the ubiquity assumption in Theorem \ref{t1} can be justified by alluding to a result of   Boshernitzan \& Chaika   \cite{BoC} about Borel-Cantelli sequences. 
According to  \cite[Theorem 2]{BoC},
 if $\{x_n: n\in\N\} \subset [0,1]$ is a sequence such that for some ball $B\subset [0,1]$ and for any $\epsilon>0$  there exists $N_{\epsilon}$ such that \eq{bc}{
\mathcal{L}\Big(B\cap \bigcup_{n\le N_{\epsilon}}B(x_n, N_{\epsilon}^{-1})\Big)\le \epsilon\cdot \mathcal{L}(B),}
 then there exists a non-increasing positive function $\psi$ such that $$
\sum_{n=1}^{\infty}\psi(n)=\infty, \ {\text{but}}\ \ \mathcal{L}\big(\big\{x\in B: |x-x_n|<\psi(n), \ {\text{i.m.}}\ n\in \N\big\}\big)=0.
$$
Though 
the negation of \equ{bc} is slightly weaker than the  regularity or ubiquity of the corresponding system, at least to some extent  it verifies the necessity of the ubiquity assumption in our main result. \end{rem}

\medskip

\noindent{\bf The Organization of the Paper.} 
We prove  Theorem \ref{t1}  in \S\ref{pf}, and   in the next three sections discuss its applications.  In \S\ref{st} we establish Theorem \ref{t2} and in fact prove a more general statement, where $[0,1]^d$ is replaced by the product of Cantor sets defined by digit restrictions. Theorem \ref{schmidt} is proved in \S\ref{lin}, and before that in \S\ref{sim}
we present a streamlined proof of Corollary \ref{corsim}.
\medskip

At the end of this section, we fix some notation. \begin{itemize}
  \item $\R_{\alpha}$: a resonant set.
\item $\widetilde{R}$: big rectangles; $R$: small rectangles.
\item $5B$ or $5R$: a ball or rectangle $B$ scaled by a factor of $5$, that is, the ball/rectangle with the same center as $B$ but with  radius or side lengths multiplied by $5$.
\item $c, c_i$: absolute constants;
\item $a\ll b$: if $a\le c b$ for an unspecified constant $c>0$; $a\asymp b$:  $a\ll b$ and $b\ll a$;
\item $r_B$: the radius of a ball $B$.
\end{itemize}

\section{Proof of the Main Result}\label{pf}
To 
estimate the measure of a limsup set from below, the following Chung-Erd\"{o}s inequality \cite{Chung} is widely used.
\begin{lem}[Chung-Erd\"{o}s inequality \cite{Chung}, see also \cite{KoS}]\label{L:PZI} Let $(\Omega, \mathcal{B}, \nu)$ be a a finite measure space, and let $\{E_n\}_{n\ge 1}$ be a sequence of measurable sets.  If $\sum_{n\ge 1}\nu(E_n)=\infty$, then $$
\nu\left(\limsup_{n\to \infty}E_n\right)\ge \limsup_{n\to \infty}\frac{ \left(\sum_{1\le n\le N}\nu(E_n)\right)^2}{\sum_{1\le i\ne j\le N}\nu(E_i\cap E_j)}.
$$
\end{lem}

The Chung-Erd\"{o}s lemma enables one to conclude the positivity of the measure of a set. In applications, to arrive at a full measure result, one often uses the Chung-Erd\"{o}s lemma restricted to a local set and then applies the following proposition.


\begin{lem}[\cite{BDV06}]\label{l2.2} Let $\Omega$ be a metric space, and let $\nu$ be a finite doubling\footnote{A measure $\nu$ on $\Omega$ is called doubling if  $\exists\, c>0$ such that for any $x\in \Omega$ and $r>0$, $
\nu\big(B(x,2r)\big)\le c\cdot \nu\big(B(x,r)\big)
$.}
 Borel measure on $\Omega$.
Let $E$ be a Borel subset of $\Omega$. Assume that there are constants $r_0$ and $c>0$ such that $$
\nu(E \cap B)\ge c\cdot \nu(B)\ \ {\text{for any ball}}\ \ B\subset \Omega \ {\text{with}}\ r_B<r_0.$$
Then $E$ has full measure in $\Omega$, i.e. $\nu(\Omega\smallsetminus E)=0.$
\end{lem}

The following $5r$-covering lemma for rectangles will be used frequently later. Generally speaking, there are no such covering lemmas
for arbitrary rectangles compared with the ones for balls; some additional assumptions on the rectangles are needed.
  In the product space 
$X =   \prod_{i=1}^d X_i$ we say that two aligned rectangles
  $$
\prod_{i=1}^d B(x_i, r_i), \ \ \prod_{i=1}^d B(y_i, \epsilon_i)
$$ are uniform in size if $$
r_i\ge \epsilon_i\ {\text{for some}} \ 1\le i\le d\Longrightarrow r_i\ge \epsilon_i, \ {\text{for all}}\ 1\le i\le d.
$$
\begin{lem} Let $(X, {\text{dist}})$ be the product of the metric spaces $(X_i, {\text{dist}_i})$ for $1\le i \le d$. Every family $\mathcal{G}$ of
aligned rectangles which are uniform in size and have bounded diameters in $X$ contains a disjoint subfamily $\mathcal{F}$ such that
$$\bigcup_{R\in \mathcal{G}}R\subset \bigcup_{R\in \mathcal{F}}5R.
$$
\end{lem}The proof applies the same idea for the classical $5r$-covering lemma for balls, so we omit the proof here. For a proof of the $5r$-covering lemma for balls one is referred to \cite[Theorem 1.2]{Hein} for the case of general
metric spaces or to \cite[Theorem 2.1]{Matt} for a constructive proof when the metric space is
boundedly compact.
\smallskip

\noindent{\em Proof of Theorem \ref{t1}}.
  We will apply the Chung-Erd\"{o}s 
  inequality to $\W(\Psi)$ restricted to an arbitrary ball. Fix a ball $B\subset X$. By the monotonicity of $\Psi$, it is trivial that \begin{align}\label{1}
\W(\Psi)\cap B&=\limsup_{n\to \infty} \left(B\cap\bigcup_{\alpha\in J_n}\prod_{i=1}^d\Delta\Big(\R_{\alpha, i}, \psi_i(\beta_{\alpha})\Big)\right)\nonumber\\
&\supset \limsup_{n\to \infty} \left(B\cap\bigcup_{\alpha\in J_n}\prod_{i=1}^d\Delta\Big(\R_{\alpha, i}, \psi_i(u_n)\Big)\right).
\end{align}
We note that this is the only point where monotonicity of $\Psi$ is used. That is why we do not require monotonicity in Theorem \ref{t2}, since (\ref{1}) will be an equality in that case.

\smallskip
  {\em Step 1}.
  For each $n$, cover the intersection 
   $$
  \frac{1}{2}B\cap \bigcup_{\alpha\in J_n}\prod_{i=1}^d\Delta\Big(\R_{\alpha, i},  \rho_i(u_n)\Big)
  $$ by the collection of rectangles in $X$ of the following form:
  $$\left\{
  \widetilde{\mathcal{R}}=\prod_{i=1}^d B\big(x_i, \rho_i(u_n)\big): x=(x_1,\dots, x_d)\in \R_{\alpha}, \ \alpha\in J_n\right\}.$$ Then  one can use the
   $5r$-covering lemma for these aligned rectangles (it is clear that the uniformity in size condition is satisfied)  
  to choose a certain subfamily $\mathcal{F}_n$ of those rectangles $\widetilde{\mathcal{R}}$. Denoting by $\mathcal{A}_n$ the collection of their centers, we can guarantee that these rectangles satisfy  the following two assumptions: \begin{itemize}
  \item (1)  Disjointness $$
  5\widetilde{\mathcal{R}}\cap 5\widetilde{\mathcal{R}}'=\left(\prod_{i=1}^d B\big(x_{i}, 5\rho_i(u_n)\big)\right)\cap \left(\prod_{i=1}^d B\big(x_{i}', 5\rho_i(u_n)\big)\right)=\emptyset, \ \ {\text{for any}}\ x\ne x'\in \mathcal{A}_n;
  $$
  \item (2) Almost packing $$
  \frac{1}{2}B\cap \bigcup_{\alpha\in J_n}\prod_{i=1}^d\Delta\big(R_{\alpha, i},  \rho_i(u_n)\big)\subset \bigcup_{\widetilde{\mathcal{R}}\in \mathcal{F}_n}5\widetilde{\mathcal{R}}=\bigcup_{x\in \mathcal{A}_n}\prod_{i=1}^d B\big(x_{i}, 5\rho_i(u_n)\big)\subset B.
  $$ Thus by a 
  measure computation argument, together with the ubiquity property applied to $\frac12B$, there is an integer $n_o(B)$ such that for all $n\ge n_o(B)$ one has $$
  \sharp \mathcal{A}_n\asymp \prod_{i=1}^d\left(\frac{r_B}{\rho_i(u_n)}\right)^{\delta_i}.
  $$ 
\end{itemize}
We will refer to the rectangles $\widetilde{\mathcal{R}}$ in $\mathcal{F}_n$ as {\em big rectangles} of level $n$.

\smallskip

{\em Step 2}. We 
intend to construct a subset of $\W(\Psi)$. Fix  a rectangle $\widetilde{\mathcal{R}}\in \mathcal{F}_n$ centered at 
$x\in \mathcal{A}_n$. Let $\R_{\alpha}$ be 
a resonant set containing $x$ for some $\alpha\in J_n$ 
(if there are several
of these $\alpha$, we only choose and fix one). Then we consider the set $$
\widetilde{\mathcal{R}}\cap \prod_{i=1}^d \Delta\big(\R_{\alpha, i}, \psi_i(u_n)\big)=\prod_{i=1}^d \Big[B\big(x_i, \rho_i(u_n)\big)\cap \Delta\big(\R_{\alpha, i}, \psi_i(u_n)\big)\Big].
$$
We can cover it by rectangles of the form $$
\mathcal{R}=\prod_{i=1}^dB\big(z_i, \psi_i(u_n)\big)
$$ with 
centers in $\R_{\alpha}$. 
Again applying the $5r$-covering lemma, we get a certain subfamily $\mathcal{C}(\widetilde{\mathcal{R}})$ of rectangles. Denoting the collection of their centers by $\mathcal{C}(x)$, we see that these rectangles satisfy the following two conditions: \begin{itemize}
  \item (1) Disjointness $$
  \prod_{i=1}^d B\big(z_{i}, 5\psi_i(u_n)\big)\cap \prod_{i=1}^d B\big(z_{i}', 5\psi_i(u_n)\big)=\emptyset  \ \ {\text{for any}}\ z\ne z' \in \mathcal{C}(x);
  $$
  \item (2) Almost packing $$
  \frac{1}{2}\widetilde{\mathcal{R}}\cap \prod_{i=1}^d \Delta\big(\R_{\alpha, i}, \psi_i(u_n)\big)\subset \bigcup_{z\in \mathcal{C}(x)}\prod_{i=1}^d B\big(z_{i}, 5\psi_i(u_n)\big)\subset \widetilde{\mathcal{R}}\cap \prod_{i=1}^d \Delta\big(\R_{\alpha, i}, 5\psi_i(u_n)\big).
  $$ Recall the $\kappa_i$-scaling property of $\R_{\alpha, i}$ for each $1\le i\le d$, so still by a measure computation argument, one has $$
  \sharp \mathcal{C}(x)\asymp \prod_{i=1}^d\left(\frac{\rho_i(u_n)}{\psi_i(u_n)}\right)^{\delta_i \kappa_i}.
  $$
\end{itemize}
We will refer to these small rectangles $\mathcal{R}$ as to {\em shrunk rectangles} of level $n$.
Then define
$$
\mathcal{E}_n=\Big\{\mathcal{R}\in  \mathcal{C}(\widetilde{\mathcal{R}}):\widetilde{\mathcal{R}}\in \mathcal{F}_n\Big\}=\Big\{\prod_{i=1}^dB\big(z_i, \psi_i(u_n)\big):z\in \mathcal{C}(x), x\in \mathcal{A}_n\Big\}
$$ and $$
E_n:=\bigcup_{\mathcal{R}\in \mathcal{E}_n}\mathcal{R}=\bigcup_{\widetilde{\mathcal{R}}\in \mathcal{F}_n}\bigcup_{\mathcal{R}\in \mathcal{C}(\widetilde{\mathcal{R}})}\mathcal{R}=\bigcup_{x\in \mathcal{A}_n}\bigcup_{z\in \mathcal{C}(x)}\prod_{i=1}^dB\big(z_i, \psi_i(u_n)\big).
$$

The process of the construction of $E_n$ can be outlined as follows: for a given ball $B$, \begin{align*}
\frac12B\overset{{\text{ubiquity}}}{--\longrightarrow}&\ \mathcal{F}_n\ {\text{or}}\ \mathcal{A}_n: \ {\text{big rectangles}}\ \widetilde{\mathcal{R}}=\prod_{i=1}^dB\big(x_i, \rho_i(u_n)\big)\\
&\overset{{\text{intersect with}}\ \Delta(\R_{\alpha}, \psi_i(u_n))}{----------\longrightarrow}\mathcal{C}(\widetilde{\mathcal{R}})\ {\text{or}}\ \mathcal{C}(x):\ {\text{shrunk rectangles}}\ {\mathcal{R}}=\prod_{i=1}^dB\big(z_i, \psi_i(u_n)\big).
\end{align*}

Clearly $$
B\cap \limsup_{n\to \infty}\left(\bigcup_{\alpha\in J_n}\prod_{i=1}^d\Delta\big(\R_{\alpha, i}, \psi_i(\beta_{\alpha})\big)\right)\supset \limsup_{n\to \infty} E_n.
$$
The 
limsup set in the right hand side is the one to which we will apply the Chung-Erd\"{o}s lemma.

At first, it is easy to see that the measure of $E_n$ can be estimated as follows: \begin{align*}
 \mu(E_n)&=\sum_{x\in \mathcal{A}_n} \sharp \mathcal{C}(x)\cdot \prod_{i=1}^d\psi_i(u_n)^{\delta_i}
  =\prod_{i=1}^d\left(\frac{r_B}{\rho_i(u_n)}\right)^{\delta_i}\cdot \prod_{i=1}^d\left(\frac{\rho_i(u_n)}{\psi_i(u_n)}\right)^{\delta_i \kappa_i}\cdot \prod_{i=1}^d\psi_i(u_n)^{\delta_i}\\
  &\asymp\mu(B)\cdot \prod_{i=1}^d\left(\frac{\psi_i(u_n)}{\rho_i(u_n)}\right)^{\delta_i(1-\kappa_i)}.
\end{align*}
So $$
\sum_{n=1}^{\infty}\mu(E_n)=\infty,
$$ and then the first condition in the Chung-Erd\"{o}s lemma is satisfied.

\smallskip

{\em Step 3}.
We estimate the 
measure of $E_m\cap E_n$ for $m<n$. Notice that \begin{align*}
  \mu(E_m\cap E_n)=\sum_{\mathcal{R}\in \mathcal{E}_m}\mu(\mathcal{R}\cap E_n)&=\sum_{x\in \mathcal{A}_m}\sum_{z\in \mathcal{C}(x)}\mu\left(\prod_{i=1}^dB\big(z_i, \psi_i(u_m)\big)\cap E_n\right)\\
  &=\sum_{x\in \mathcal{A}_m}\sum_{z\in \mathcal{C}(x)}\mu\left(\prod_{i=1}^dB\big(z_i, \psi_i(u_m)\big)\cap \bigcup_{x'\in \mathcal{A}_n}\bigcup_{z'\in \mathcal{C}(x')}\prod_{i=1}^dB\big(x_i', \psi_i(u_n)\big)\right).
\end{align*}

Since all the rectangles in $\mathcal{E}_n$ are of the same size, we need only estimate the number of elements in $\mathcal{E}_n$ which can intersect a given  element in $\mathcal{E}_m$.  So fix an arbitrary rectangle $\mathcal{R}$ in $\mathcal{E}_m$ which is of the form $$
\mathcal{R}=\prod_{i=1}^dB\big(z_i, \psi_i(u_m)\big).
$$ Recall the construction of $E_n$. At first, we estimate the number of 
big rectangles in $\F_n$ which can intersect $\mathcal{R}$. Remember that all the big rectangles in $\F_n$ are of the same side lengths $\big(\rho_1(u_n),\dots, \rho_d(u_n)\big)$.
Let $$
I_1:=\big\{1\le i\le d: \psi_i(u_m)\ge \rho_i(u_n)\big\},\ \ I_2:=\big\{1\le i\le d: \psi_i(u_m)< \rho_i(u_n)\big\}.
$$
Define an enlarged body of the rectangle $\mathcal{R}$: $$
H:=\prod_{i=1}^dB(z_i, 3\epsilon_i), \ {\text{where}}\ \epsilon_i:=\left\{
                                                              \begin{array}{ll}
                                                                \psi_{i}(u_m), & \hbox{for $i\in I_1$;} \\
                                                                \rho_i(u_n), & \hbox{for $i\in I_2$.}
                                                              \end{array}
                                                            \right.
$$ Thus all the big rectangles in $\F_n$ which can intersect $\mathcal{R}$ are contained in $H$. Since these big rectangles in $\F_n$ are disjoint, a measure computation argument gives that the number of big rectangles in $\F_n$ which can possibly intersect the rectangle $\mathcal{R}$ is bounded from above by
\begin{equation*}\label{e1}
\prod_{i\in I_1}\left(\frac{\psi_i(u_m)}{\rho_i(u_n)}\right)^{\delta_i}.
\end{equation*}

Secondly, fix a center $x'\in \mathcal{A}_n$ or, equivalently, a big rectangle $\widetilde{\mathcal{R}}_n=\prod_{i=1}^dB\big(x_i', \rho_i(u_n)\big)$ in $\mathcal{F}_n$ which has non-empty intersection with the rectangle $\mathcal{R}$. We consider the number $L$ of shrunk rectangles in $\mathcal{E}_n$ which can intersect the set $$
\mathcal{R}\cap \widetilde{\mathcal{R}}_n=\prod_{i=1}^dB\big(z_i, \psi_i(u_m)\big)\cap \prod_{i=1}^dB\big(x_i', \rho_i(u_n)\big).
$$
Clearly all these $L$ shrunk rectangles are contained in \begin{equation}\label{e2}
\prod_{i=1}^dB\big(z_i, 2\psi_i(u_m)\big)\cap \prod_{i=1}^dB\big(x_i', 2\rho_i(u_n)\big)\cap \prod_{i=1}^d\Delta\big(\R_{\alpha, i}, \psi_i(u_n)\big),
\end{equation} where $\alpha\in J_n$ is the index for which $x'=(x_1',\dots, x_d')$ 
lies in $\R_{\alpha}$. Then by a measure computation argument, the number $L$ can be estimated as
\begin{align*}
  L&\le \frac{{\text{the measure of the set (\ref{e2})}}}{{\text{the measure of a shrunk rectangle}}}.\end{align*}
Thus by the $\kappa$-scaling property of $\R_{\alpha}$, it follows that \begin{align*}
L&\ll \frac{\displaystyle \prod_{i\in I_1}\rho_i(u_n)^{\delta_i\kappa_i}\cdot \psi_i(u_n)^{\delta_i(1-\kappa_i)}\cdot \prod_{i\in I_2}\psi_i(u_m)^{\delta_i \kappa_i}\cdot \psi_i(u_n)^{\delta_i(1-\kappa_i)}}
  {\displaystyle \prod_{i=1}^d\psi_i(u_n)^{\delta_i}}\\
  &=\prod_{i\in I_1}\left(\frac{\rho_i(u_n)}{\psi_i(u_n)}\right)^{\delta_i\kappa_i}\cdot \prod_{i\in I_2}\left(\frac{\psi_i(u_m)}{ \psi_i(u_n)}\right)^{\delta_i\kappa_i}.
\end{align*}

At last, we can estimate the measure of $E_m\cap E_n$. More precisely, \begin{align*}
  \mu(E_m\cap E_n)&\le \sum_{{\mathcal{R}}\in \mathcal{E}_m}\prod_{i\in I_1}\left(\frac{\psi_i(u_m)}{\rho_i(u_n)}\right)^{\delta_i}\cdot L\cdot \prod_{i=1}^d\psi_i(u_n)^{\delta_i}\\
  &\ll\sum_{{\mathcal{R}}\in \mathcal{E}_m}\prod_{i\in I_1}\left(\frac{\psi_i(u_m)}{\rho_i(u_n)}\right)^{\delta_i}\cdot \prod_{i\in I_1}\left(\rho_i(u_n)^{\delta_i\kappa_i}\cdot \psi_i(u_n)^{\delta_i(1-\kappa_i)}\right)\cdot \prod_{i\in I_2}\left(\psi_i(u_m)^{\delta_i\kappa_i}\cdot \psi_i(u_n)^{\delta_i(1-\kappa_i)}\right).\end{align*}
  Recall the number of the elements in $\mathcal{E}_m$. It follows that
  \begin{align}
  \mu(E_m\cap E_n) &\ll \left[\prod_{i=1}^d\left(\frac{r_B}{\rho_i(u_m)}\right)^{{\delta_i}}\cdot \prod_{i=1}^d\left(\frac{\rho_i(u_m)}{\psi_i(u_m)}\right)^{\delta_i\kappa_i}\right]\cdot \prod_{i\in I_1}\left(\frac{\psi_i(u_m)}{\rho_i(u_n)}\right)^{\delta_i}\cdot \nonumber\\
   &\qquad \qquad \qquad \qquad \qquad \qquad \prod_{i\in I_1}\left(\rho_i(u_n)^{\delta_i\kappa_i}\cdot \psi_i(u_n)^{\delta_i(1-\kappa_i)}\right)\cdot \prod_{i\in I_2}\left(\psi_i(u_m)^{\delta_i\kappa_i}\cdot \psi_i(u_n)^{\delta_i(1-\kappa_i)}\right)\nonumber\\
  &=\mu(B)\cdot \prod_{i\in I_1}\left(\frac{\psi_i(u_m)}{\rho_i(u_m)}\right)^{\delta_i(1-\kappa_i)}\cdot \prod_{i\in I_2}\left(\frac{\rho_i(u_n)}{\rho_i(u_m)}\right)^{\delta_i(1-\kappa_i)}\cdot \prod_{i=1}^d\left(\frac{\psi_i(u_n)}{\rho_i(u_n)}\right)^{\delta_i(1-\kappa_i)}\label{e3}\\
&=\mu(B)\cdot \prod_{i=1}^d\left(\frac{\psi_i(u_m)}{\rho_i(u_m)}\right)^{\delta_i(1-\kappa_i)}\cdot \prod_{i\in I_2}\left(\frac{\psi_i(u_n)}{\psi_i(u_m)}\right)^{\delta_i(1-\kappa_i)}\cdot \prod_{i\in I_1}\left(\frac{\psi_i(u_n)}{\rho_i(u_n)}\right)^{\delta_i(1-\kappa_i)}\label{e4}.
\end{align}
\noindent
When $I_2=\emptyset$,
\begin{align*}
  \mu(E_m\cap E_n)\ll\mu(B)\cdot \prod_{i\in I_1}\left(\frac{\psi_i(u_m)}{\rho_i(u_m)}\right)^{\delta_i(1-\kappa_i)}\cdot \left(\prod_{i=1}^d\frac{\psi_i(u_n)}{\rho_i(u_n)}\right)^{\delta_i(1-\kappa_i)}\asymp \mu(E_m)\cdot\mu(E_n)\cdot\mu(B)^{-1}.
\end{align*}
When $I_2\ne \emptyset$, \begin{itemize}\item if $\rho$ is $\lambda$-regular, then by (\ref{e3}) it follows that\begin{align*}
  \mu(E_m\cap E_n)&\ll\mu(B)\cdot \prod_{i\in I_1}\left(\frac{\psi_i(u_m)}{\rho_i(u_m)}\right)^{\delta_i(1-\kappa_i)}\cdot \prod_{i\in I_2} \lambda^{(n-m)\delta_i(1-\kappa_i)}\cdot \left(\prod_{i=1}^d\frac{\psi_i(u_n)}{\rho_i(u_n)}\right)^{\delta_i(1-\kappa_i)}\\
&\le\mu(B)\cdot \lambda^{(n-m)\epsilon}\cdot  \left(\prod_{i=1}^d\frac{\psi_i(u_n)}{\rho_i(u_n)}\right)^{\delta_i(1-\kappa_i)}\ \ \ \ \big({\text{by}}\ \psi_i(u_m)\le \rho_i(u_m)\big)\\
  & \asymp\mu(E_n)\cdot \lambda^{(n-m)\epsilon}.
\end{align*}
\item if $\Psi$ is $\lambda$-regular, then by (\ref{e4}) it follows that\begin{align*}
  \mu(E_m\cap E_n)&\ll\mu(B)\cdot \prod_{i=1}^d\left(\frac{\psi_i(u_m)}{\rho_i(u_m)}\right)^{\delta_i(1-\kappa_i)}\cdot \prod_{i\in I_2} \lambda^{(n-m)\delta_i(1-\kappa_i)}\cdot \left(\prod_{i\in I_1}\frac{\psi_i(u_n)}{\rho_i(u_n)}\right)^{\delta_i(1-\kappa_i)}\\
&\le\mu(B)\cdot  \left(\prod_{i=1}^d\frac{\psi_i(u_m)}{\rho_i(u_m)}\right)^{\delta_i(1-\kappa_i)}\cdot \lambda^{(n-m)\epsilon}\\
  & \asymp \mu(E_m)\cdot \lambda^{(n-m)\epsilon}.
\end{align*}\end{itemize}

{\em Step 4.} Finally, to apply the Chung-Erd\"{o}s lemma, we calculate the correlations. \begin{itemize}
  \item if $\rho$ is $\lambda$-regular,
\begin{align*}
  \sum_{1\le m<n\le N}\mu(E_m\cap E_n)&=\sum_{n=1}^N\sum_{m=1}^{n-1}\mu(E_m\cap E_n)\\
&\ll \sum_{n=1}^N\sum_{m=1}^{n-1}\left(\frac{1}{\mu(B)}\cdot \mu(E_m)\cdot\mu(E_n)+\mu(E_n)\cdot \lambda^{(n-m)\epsilon}\right)\\
&\ll \frac{1}{\mu(B)}\cdot \left(\sum_{n=1}^N\mu(E_n)\right)^2+\sum_{n=1}^N\mu(E_n).
\end{align*}
\item if $\Psi$ is $\lambda$-regular,\begin{align*}
  \sum_{1\le m<n\le N}\mu(E_m\cap E_n)&=\sum_{m=1}^N\sum_{n=m+1}^{N}\mu(E_m\cap E_n)\\
&\ll \sum_{m=1}^N\sum_{n=m+1}^{N}\left(\frac{1}{\mu(B)}\cdot \mu(E_m)\cdot\mu(E_n)+\mu(E_m)\cdot \lambda^{(n-m)\epsilon}\right)\\
&\ll \frac{1}{\mu(B)}
 \left(\sum_{m=1}^N\mu(E_m)\right)^2+\sum_{m=1}^N\mu(E_m).
\end{align*}
\end{itemize}
In a summary, we have shown $$
\sum_{1\le m< n\le N}\mu(E_m\cap E_n)\ll\mu(B)^{-1}\left(\sum_{1\le n\le N}\mu(E_n)\right)^2+\sum_{1\le n\le N}\mu(E_n).
$$ By the Chung-Erd\"{o}s lemma, it follows that 
$$\mu\big(\W(\Psi)\cap B\big)\gg \mu(B).$$

Clearly, the measure $\mu$ is Ahlfors regular, hence doubling. Thus by Lemma \ref{l2.2}, one concludes that $\W(\Psi)$ is of full measure.
\hfill $\Box$

%

\section{A Shrinking Target Problem} \label{st}

Here we are going to prove a statement slightly more general than Theorem \ref{t2}.
Let $b_1,\dots, b_d\ge 2$ be a $d$-tuple of integers. Let $$\Lambda_{i}\subset \{0,1,\dots, b_i-1\}, \ {\text{with}} \ \ \sharp \Lambda_{i}\ge 2 \ \ {\text{for}}
\ \ 1\le i\le d.$$ Then let $\C_i$ be the Cantor sets defined by the iterated function systems $$
\Big\{g_{b_i,k}(x)=\frac{x+k}{b_i}, \ x\in [0,1], \  k\in \Lambda_{i}\Big\}.
$$ The natural Cantor measure $\mu_i$ supported on $\C_i$ is Ahlfors regular \cite{Hutch} with exponent $\delta_i=\frac{\log \sharp \Lambda_i}{\log b_i}$.

For $d$ positive functions $\psi_i: \mathbb{R}_{+}\to \mathbb{R}_+$ ($1\le i\le d$), define
$$M_c(\psi):=\Big\{(x_1,\dots, x_d)\in \prod_{i=1}^d\mathcal{C}_i: \|b_i^nx_i-x_{o,i}\|<\psi_i(n), \ 1\le i\le d,\ {\text{i.m.}}\ n\in \N\Big\}, \ \ x_o\in \prod_{i=1}^d\mathcal{C}_i.
$$

We use the symbolic representations of the points $x_i$ in $\C_i$. For each $\w_i=(\epsilon_1,\dots, \epsilon_n)\in \Lambda^n_{i}$ with $n\ge 1$, write $$
I_{n,b_i}(\w_i):=g_{b_i, \epsilon_1}\circ g_{b_i, \epsilon_2}\circ \cdots \circ g_{b_i, \epsilon_n}[0,1], \ \ x_{i}(\w_i)=\frac{\epsilon_1}{b_i}+\cdots+\frac{\epsilon_n+x_{o,i}}{b_i^n},
$$ in other words $I_{n,b_i}(\w_i)$ is an $n$th order cylinder with respect to $\mathcal{C}_i$, and $x_i(\w_i)$ is the $n$th inverse image of $x_{o,i}$ in $I_{n,b_i}(\w_i)$. Note that for any $\w_i$ there is an inverse image of $x_{o,i}$ in  $I_{n,b_i}(\w_i)$ and the length of $I_{n,b_i}(\w_i)$ is $b_i^{-n}$.

Clearly the set $M_{c}(\psi)$ can be rewritten as \begin{align*}
M_c(\psi)&=\Big\{x\in \prod_{i=1}^d\mathcal{C}_i: \big|x_i-x_{i}(\w_i)\big|<\frac{\psi_i(n)}{b_i^n}, \ \w_i\in \Lambda^n_{i}, 1\le i\le d, \ {\text{i.m.}}\ n\in \N\Big\}
\end{align*}

Thus one has\begin{itemize}\item the index set $J$: $$
J=\Big\{\alpha=(\w_1,\dots, \w_d)\in \prod_{i=1}^d\Lambda_{i}^n: n\ge 1\Big\};
$$
\item the resonant sets $\R_{\alpha}$: $$\R_{\alpha}=\big(x_{1}(\w_1),\dots, x_{d}(\w_d)\big) \  \ \ {\text{for}}\ \alpha=(\w_1,\dots, \w_d);$$

\item the weight function $\beta_{\alpha}$: $$\beta_{\alpha}=n, \ \ {\text{for}}\ \alpha=(\w_1,\dots, \w_d)\in \prod_{i=1}^d\Lambda^n_{i};$$

\item the ubiquitous function $\rho_i$: $$\rho:\mathbb{R}_+\to \mathbb{R}_+: n\to b_i^{-n};$$

\item the sequences $$
\ell_n=u_n=n, \ n\ge 1.
$$
\end{itemize}
\begin{pro}\label{p4} The pair
$(\{\R_{\alpha}\}_{\alpha\in J}, \beta)$ is a ubiquitous system for rectangles with respect to the function $\rho$ and the sequences $\{\ell_n, u_n\}_{n\ge 1}$.
Meanwhile, the $\kappa$-scaling property holds with $\kappa=0$.\end{pro}
\begin{proof} This is rather simple since $$
\bigcup_{\ell_n\le \beta_{\alpha}\le u_n}\prod_{i=1}^d\Delta\big(\R_{\alpha,i}, \rho_i(u_n)\big)=\bigcup_{\w_i\in \Lambda^n_{i}, 1\le i\le d}\ \prod_{i=1}^dB\big(x_{i}(\w_i), b_i^{-n}\big)=\prod_{i=1}^d\mathcal{C}_i.
$$
\end{proof}
It is trivial that $\rho$ is $\lambda$-regular. Then by Theorem \ref{t1}, it follows that $$
\mu\big(M_c(\psi)\big)=1, \ {\text{or}}\ 0 \Longleftrightarrow \sum_{n=1}^{\infty}\prod_{i=1}^d\psi_i(n)^{\delta_i}=\infty\ {\text{or}}\ <\infty,
$$ where the convergence part follows easily from the Borel-Cantelli lemma.

If we choose $\Lambda_i = \{0,\dots,b_{i-1}\}$ for all $i$, then one has $M_{c}(\psi) =  \mathfrak{S}({\psi})$ and $\mu$ is the Lebesgue measure. Then Theorem \ref{t2} follows.

\section{Simultaneous Diophantine approximation} \label{sim}

In this section we apply 
Theorem \ref{t1} to simultaneous Diophantine approximation, establishing Corollary \ref{corsim} as a warm-up before proving Theorem \ref{schmidt}.
 Recall that we are given  a $d$-tuple $\varphi = \{\varphi_i\}_{1\le i\le d}$ of non-increasing positive functions defined on $\mathbb{Z}_{\ge 0}$ with $$\varphi_i(q)\to 0, \ {\text{as}}\ q\to \infty,$$
 and our goal is to show that the Lebesgue measure of
 $${W}(\varphi)=\big\{x\in [0,1]^d: \|qx_i\|<\varphi_i(q), \ 1\le i\le d, \ \ {\text{i.m.}}\ q\in \mathbb{N}\big\}.
$$
is zero or one according to$$
\sum_{q=1}^{\infty}\prod_{i=1}^d\varphi_i(q)<\infty\ {\text{or}}\ =\infty.
$$


\begin
{proof}[Proof of Corollary \ref{corsim}] 
 First observe that the following conditions can be assumed without loss of generality:
\begin{itemize}
  \item for all $q\gg 1$, \begin{equation}\label{f6}
q\prod_{i=1}^d\varphi_i(q)\le 1
\end{equation} (otherwise by Minkowski's  theorem  ${W}(\varphi)=[0,1]^d$);

\item for all $1\le i\le d$, \begin{equation}\label{f8}
\varphi_i(q)\ge q^{-1-\frac1{2d}}, \ {\text{and so}}\ q^{d+1}\cdot\prod_{i=1}^d\varphi_i(q)\ge q^{1/2}\to \infty.
\end{equation} Otherwise, we define $$
\overline{\varphi}_i(q)=\max\{\varphi_i(q), q^{-1-\frac1{2d}}\},
$$ and consider the set ${{W}}(\overline{\varphi})$. For any $x\in {{W}}(\overline{\varphi})\smallsetminus {W}(\varphi)$, one has, for some index $1\le i\le d$, $$
\|qx_i\|<q^{-1-\frac1{2d}}, \ {\text{for infinitely many}}\ q\in \mathbb{N}.
$$ Thus, by the Borel-Cantelli Lemma, the above set 
is 
Lebesgue null, 
and 
it follows that $$
\mathcal{L}\big({{W}}(\overline{\varphi})\big)=\mathcal{L}\big({W}(\varphi)\big).
$$
\end{itemize}

Now we will check that all the conditions in Theorem \ref{t1} are satisfied by a suitable choice of the ubiquitous function $\rho$.
\begin{itemize}
\item the index and resonant sets: \begin{align*}
J&=\big\{(q, p_1,\dots, p_d): q\in \mathbb{N}, \  0\le p_i\le q, \ 1\le i\le d\big\},\\
\R_{\alpha}&=\left(\frac{p_1}{q},\dots, \frac{p_d}{q}\right)  \ \ {\text{and}}\ \ \beta_{\alpha}=q,\ \ {\text{for}}\ \ \alpha=(q, p_1,\dots, p_d);
\end{align*}
  \item $\mu_i = \mathcal{L}$, which is Ahlfors regular with $\delta_i=1$;
  \item $\kappa$-scaling: $\kappa_i=0$ since $\R_{\alpha, i}$ are points for all $1\le i\le d$ and $\alpha\in J$;
  \item the approximating function: $$
  \psi_i(q)=\frac{\varphi_i(q)}{q}, \ \ 1\le i\le d;
  $$
  \item the ubiquitous fucntion: let \begin{equation}\label{f7}
  \rho_i(q)=\frac{\varphi_i(q)}{q}\cdot \left(q\prod_{i=1}^d\varphi_i(q)\right)^{-1/d},\ 1\le i\le d;
  \end{equation}
  \item $u_n=M^n$ and $\ell_n=M^{n-1}$ with $M\ge 2^{2d+3}$.
\end{itemize}

\begin{lem}[Ubiquity for rectangles]\label{l4.2} With the notation above, the 
pair $(\{\R_{\alpha}\}_{\alpha\in J}, \beta)$ is a ubiquitous system with respect to the function $\rho$ and the sequences $\{\ell_n, u_n\}_{n\ge 1}$.
\end{lem}
\begin{proof} We will give a detailed proof for the case of linear forms later, see Lemma \ref{l5.3}.
Then Lemma \ref{l4.2} follows by taking $h=1$ and $\Psi(q)=q$ there.
\end{proof}

\begin{lem} For all $1\le i\le d$, $$
\psi_i(q)\le \rho_i(q), \ \ \ \rho_i(q)\to 0,\ \ \psi_i\ {\text{is $\lambda$-regular}}.
$$
\end{lem}
\begin{proof}
  The first 
  inequality is clear by (\ref{f6}) and (\ref{f7}). For the third 
  condition, by the monotonicity of $\varphi_i$, one has $$
\psi_i(M^{n+1})=\frac{\varphi_i(M^{n+1})}{M^{n+1}}\le \frac{\varphi_i(M^{n})}{M^{n+1}}=\frac{1}{M}\cdot \psi_i(M^n).
$$ For the second one, replacing $\varphi_i(q)$ by 1 in (\ref{f7}), it suffices to show that $$
q^{d+1}\prod_{i=1}^d\varphi_i(q)\to \infty, \ {\text{as}}\ q\to \infty,
$$ which follows from (\ref{f8}).
\end{proof}

At last, we notice that \begin{align*}
  \sum_{n=1}^{\infty}\prod_{i=1}^d\frac{\psi_i(u_n)}{\rho_i(u_n)}=\sum_{n=1}^{\infty}M^n\prod_{i=1}^d\varphi_i(M^n)\asymp \sum_{q=1}^{\infty}\prod_{i=1}^d\varphi_i(q).
\end{align*}
Thus all the conditions in Theorem \ref{t1} are satisfied, and then it yields that $$
\mathcal{L} \big({W}(\Phi)\big)=1  \ {\text{if}}\ \ \sum_{q=1}^{\infty}\prod_{i=1}^d\varphi(q)=\infty.
$$

The convergence part of Corollary \ref{corsim} follows from the convergence part of the Borel-Cantelli Lemma, which finishes the proof of the corollary.
\end{proof}

\section{Systems of Linear Forms} \label{lin}

In this section, we prove Theorem {\ref{schmidt}} by applying Theorem \ref{t1}. The main task is to find the suitable ubiquitous function. Recall that
$\{\varphi_i\}_{1\le i\le d}$ are $d$ non-increasing positive functions defined on $\mathbb{Z}_{\ge 0}$ with $$\varphi_i(u)\to 0, \ {\text{as}}\ u\to \infty,$$
 and $\{\Phi_k\}_{1\le k\le h}$ are $h$ non-decreasing integer valued functions with
$$\Phi_k:\N\to \N,\ \ \Phi_k(u)\to \infty, \ {\text{as}}\ u\to \infty.$$
Recall that we are considering  the set
\begin{align*}
 W(\varphi,\Phi)&= \bigg\{A\in [0,1]^{dh}: {\text{the system}} \left\{
                        \begin{array}{ll}
                          \|A_i\q\|<\varphi_i(u), 1\le i\le d, & \\
                          |q_k|\le\Phi_k(u), \ \ \, 1\le k\le h, &
                        \end{array}
                      \right. \\
                     & \qquad \qquad\qquad\qquad\qquad\qquad {\text{has a solution in $\q\in \Z^{h}\smallsetminus \{0\}$}} \ {\text{for infinitely many}}\ u\in \N
\bigg\}\\
&=\limsup_{\q\in \Z^h}E_{\q}(\varphi, \Phi),
\end{align*}
where
$$E_{\q}(\varphi, \Phi) :=\left\{A\in [0,1]^{dh}: \|A_i\q\|<\varphi_i\big(\max\{\Phi_1^{-1}(|q_1|^+),\dots, \Phi_h^{-1}(|q_h|^+)\}\big), \ 1\le i\le d, \ {\text{i.m.}}\ \q\in \Z^h\right\}.$$



 We begin with a technical lemma which enables us to choose the ubiquitous functions fulfilling the conditions that $\rho_i(u)\to 0$ and $\psi_i\le \rho_i$.
\begin{lem}\label{l3.1}  Assume that there exists $M>1$ such that for all $n\gg 1$, $$c_1\Phi_k(M^n)\le \Phi_k(M^{n+1})\le c_2 \Phi_k(M^n), \ 1\le k\le h$$ for some absolute 
 constants $c_1, c_2>1$. 
 Then, without changing the measure of $W(\varphi, \Phi)$,
we can assume, without loss of generality, that
 \begin{align}\label{e6}\Big(\max_{1\le k\le h}\Phi_k(u)\Big)^d\cdot \prod_{k=1}^h\Phi_k(u)\cdot\prod_{i=1}^d\varphi_i(u)\to \infty  \ {\text{when}}\ u\to \infty.\end{align}
\end{lem}
\begin{proof} We define a new collection of functions $\widetilde{\varphi}_i$ for $1\le i\le d$ satisfying 
condition (\ref{e6}) and $$
\mathcal{L}\big(W(\widetilde{\varphi}, \Phi)\big)= \mathcal{L}\big(W({\varphi}, \Phi)\big).
$$
Fix an increasing function $f$, say 
$f(u)=\big(\max_{1\le k\le h}\Phi_k(u)\big)^{\epsilon}$ for example, which tends to infinity with a slow speed as $u\to \infty$, and check that $$
\frac{f(u)}{\Big(\max_{1\le k\le h}\Phi_k(u)\Big)^d\cdot \prod_{k=1}^h\Phi_k(u)}\ \ {\text{is decreasing with respect to}}\ u.
$$

Partition 
$\N$ into two classes: $$
\mathcal{N}_1=\Big\{u\in \N: \Big(\max_{1\le k\le h}\Phi_k(u)\Big)^d\cdot \prod_{k=1}^h\Phi_k(u)\cdot\prod_{i=1}^d\varphi_i(u)\ge f(u)\Big\}
$$ and its complement.

We can assume $\mathcal{N}_1$ to be non-empty 
by redefining
$\varphi_i(1)$ and $\Phi_i(1)$ so that they are 
large enough. 
Let $u_0$ be the smallest element in $\mathcal{N}_1$. We define a new collection of functions $\widetilde{\varphi}_i$ for $u\ge u_0$ with $1\le i\le d$ inductively as follows:
\begin{enumerate}
\item For $u=u_0$, define $$
\widetilde{\varphi}_i(u)=\varphi_i(u), \  \ {\text{for all}}\ 1\le i\le d;
$$

\item Let $u=u_0+1$.
\begin{itemize}
  \item if $u\in \mathcal{N}_1$, define $$
\widetilde{\varphi}_i(u)=\varphi_i(u), \  \ {\text{for all}}\ 1\le i\le d;
$$
\item if $u\not\in \mathcal{N}_1$, we increase the value of $\varphi_i(u)$ as follows. Define a function $$
G(t)=\prod_{i=1}^d\Big(t\widetilde{\varphi}_i(u_0)+(1-t)\varphi_i(u)\Big),\  \ {\text{for}}\ t\in [0,1].
$$ Then \begin{align*}
G(1)=\prod_{i=1}^d\widetilde{\varphi_i}(u_0)&\ge \frac{f(u_0)}{\Big(\max_{1\le k\le h}\Phi_k(u_{0})\Big)^d\cdot\prod_{k=1}^h\Phi_k(u_{0})}\\
&\ge \frac{f(u)}{\Big(\max_{1\le k\le h}\Phi_k(u)\Big)^d\cdot\prod_{k=1}^h\Phi_k(u)}\ge \prod_{i=1}^d{\varphi_i}(u)=G(0),
\end{align*}where the last inequality holds because $u\not\in \mathcal{N}_1$. So there exists some $t^*\in [0,1]$ such that $$
G(t^*)=\frac{f(u)}{\Big(\max_{1\le k\le h}\Phi_k(u)\Big)^d\cdot\prod_{k=1}^h\Phi_k(u)}.
$$ Thus define $$
\widetilde{\varphi_i}(u)=t^*\widetilde{\varphi}_i(u_0)+(1-t^*)\varphi_i(u), \ \ 1\le i\le d.
$$ It is clear that \begin{equation}\label{g1}
\widetilde{\varphi_i}(u_0)\ge \widetilde{\varphi_i}(u)\ge \varphi_i(u).
\end{equation}
\end{itemize}

\item Assume that $\widetilde{\varphi}_i(u')$ for all $1\le i\le d$ have been defined. Then for 
$u'' = u'+1$ the process is the same with the role of $\widetilde{\varphi_i}(u_0)$ and $\varphi_i(u)$ replaced by $\widetilde{\varphi_i}(u')$ and $\varphi_i(u'')$ respectively.\end{enumerate}

To summarize, for the new functions $\widetilde{\varphi_i}(u)$ for $1\le i\le d$  one has: \begin{itemize}
  \item by (\ref{g1}), 
  $$\widetilde{\varphi_i}\  \ {\text{is decreasing,  and}},\ \   \widetilde{\varphi_i}\ge \varphi_i;$$

\item for $u\in \mathcal{N}_1$, \begin{equation}\label{f9}
\widetilde{\varphi_i}(u)=\varphi_i(u);
\end{equation}

\item for $u\not\in \mathcal{N}_1$, $$
\Big(\max_{1\le k\le h}\Phi_k(u)\Big)^d\cdot \prod_{k=1}^h\Phi_k(u)\cdot \prod_{i=1}^d\widetilde{\varphi_i}(u)=f(u).
$$
\end{itemize}

Finally, we consider the measure of the set $W(\widetilde{\varphi}, \Phi)$. Write $$
\mathcal{M}_1=\big\{\q=(q_1,\dots, q_h)\in \Z^h: u=\max\{\Phi_1^{-1}(|q_1|^+),\dots, \Phi_h^{-1}(|q_h|^+)\}\in \mathcal{N}_1\big\}.
$$ By (\ref{f9}), one sees that \begin{align*}
 W(\widetilde{\varphi}, \Phi)&=\limsup_{\q\in \mathcal{M}_1}E_{\q}(\widetilde{\varphi}, \Phi)\cup \limsup_{\q\not\in \mathcal{M}_1}E_{\q}(\widetilde{\varphi}, \Phi)\\
&\subset W({\varphi}, \Phi)\cup \limsup_{\q\not\in \mathcal{M}_1}E_{\q}(\widetilde{\varphi}, \Phi)
\end{align*}
We claim that the second set is of measure zero by the convergence part of the Borel-Cantelli lemma, which results in $$
\mathcal{L}\big(W(\widetilde{\varphi}, \Phi)\big)= \mathcal{L}\big(W({\varphi}, \Phi)\big)
$$ as wanted. More precisely, \begin{align*}
  \sum_{\q\not\in \mathcal{M}_1}\mathcal{L}\big(E_{\q}(\widetilde{\varphi}, \Phi)\big)
&=\sum_{u\not\in \mathcal{N}_1}\sum_{\q\in \Z^h: u=\max\{\Phi_1^{-1}(|q_1|^+),\dots, \Phi_h^{-1}(|q_h|^+)\}}\ \mathcal{L}\big(E_{\q}(\widetilde{\varphi}, \Phi)\big)\\
&=\sum_{u\not\in \mathcal{N}_1}\sum_{\q\in \Z^h: u=\max\{\Phi_1^{-1}(|q_1|^+),\dots, \Phi_h^{-1}(|q_h|^+)\}}\ \prod_{i=1}^d\widetilde{\varphi}_i(u)\\
&\le\sum_{u\in \N}\sum_{\q\in \Z^h: u=\max\{\Phi_1^{-1}(|q_1|^+),\dots, \Phi_h^{-1}(|q_h|^+)\}}\frac{f(u)}{\Big(\max_{1\le k\le t}\Phi_k(u)\Big)^d\cdot\prod_{k=1}^h\Phi_k(u)}.\end{align*}
By the monotonicity of the terms in the summation and by dividing the integers $u\in \N$ into $M$-adic blocks, one has
\begin{align*}
\sum_{\q\not\in \mathcal{M}_1}\mathcal{L}\big(E_{\q}(\widetilde{\varphi}, \Phi)\big)
&\le 2^h \sum_{t=0}^{\infty}\prod_{k=1}^h\Phi_k(M^{t+1}) \cdot \frac{1}{\big(\max_{1\le k\le h}\Phi_k(M^{t})\big)^{d-\epsilon}\cdot\prod_{k=1}^h\Phi_k(M^{t})}\\
&\le 2^{h}\cdot c_2^h\cdot \sum_{t=0}^{\infty}{\big(c_1^{-t}\big)^{d-\epsilon}}<\infty,\end{align*}
where the second inequality uses the assumptions posed on $\Phi$.
Thus, by the Borel-Cantelli lemma, $
  \limsup_{\q\not\in \mathcal{M}_1}\mathcal{L}\big(E_{\q}(\widetilde{\varphi}, \Phi)\big)=0.
$
\end{proof}

We are now ready to proceed with Theorem \ref{schmidt}.
\begin{proof}[Proof of Theorem \ref{schmidt}]
By Lemma \ref{l3.1}, we can assume that the functions $\varphi, \Phi$ satisfy the conclusion given there. Moreover, without loss of generality we can assume that \begin{align}\label{g2}
\prod_{i=1}^d\varphi_i(u)\cdot \prod_{k=1}^h\Phi_k(u)\le 1, \ {\text{for all}}\ u\in \N
\end{align} otherwise, by Minkowski's convex body theorem, it is trivial that $W(\varphi, \Phi)$ is of full measure.

Now let us check that all the conditions in Theorem \ref{t1} are satisfied.

\begin{itemize}
\item The index set $J$:
$$\alpha=(q_1,\dots, q_h, p_1,\dots, p_d): \q\in \Z^h, \ |p_i|\le h\cdot \max_{1\le k\le h}{|q_k|}.$$

\item The weight function:
$$\beta_{\alpha}=\max\big\{\Phi_1^{-1}(|q_1|^+),\dots, \Phi_h^{-1}(|q_h|^+)\big\}, \ {\text{for}}\ \alpha=(q_1,\dots, q_h, p_1,\dots, p_d).$$

\item Resonant sets:
$$\R_{\alpha}=\prod_{i=1}^d\big\{A_i: A_i\q=p_i\big\}, \  \ {\text{for}}\ \alpha=(q_1,\dots, q_h, p_1,\dots, p_d).$$

\item $u_n=M^n$ for some integer $M\ge 2^{2d+3}$, and $\ell_n$ is defined later in (\ref{ln}).

\item The ubiquitous function: \begin{equation}\label{f10}
\rho_i(u)=\frac{M\cdot \varphi_i(u)}{\max_{1\le k\le h}\Phi_k(u)}\cdot \left(\prod_{i=1}^d\varphi_i(u)\cdot \prod_{k=1}^h\Phi_k(u)\right)^{-1/d}, \ {\text{for all}} \ 1\le i\le d.
 \end{equation} We will show in Lemma \ref{l5.3} that $(\{\R_{\alpha}\}_{\alpha\in J}, \beta)$ forms a ubiquitous system with respect to the ubiquitous function $\rho$ and the sequences $\{\ell_n, u_n\}_{n\ge 1}$.

\item The approximating function: \begin{equation}\label{f11}
\psi_i(u)=\frac{1}{h}\cdot \frac{\varphi_i(u)}{\max_{1\le k\le h}\Phi_k(u)}, \ {\text{for all}} \ 1\le i\le d.
\end{equation}
Note that for any $A$ such that $$
A\in \prod_{i=1}^d\Delta\left(\R_{\alpha, i}, \frac{1}{h}\cdot\frac{\varphi_i(u)}{\max_{1\le k\le h}\Phi_k(u)}\right)\ \ {\text{with}}\ \ \alpha=(q_1,\dots,q_h, p_1,\dots, p_d), \ \beta_{\alpha}=u,
$$  one has $$
|A_i\q-p_i|<|\q|\cdot \frac{\varphi_i(u)}{\max_{1\le k\le h}\Phi_k(u)}\le \varphi_i(u), \ {\text{for}}\ 1\le i\le d.
$$
 In other words, $$
\bigcup_{\alpha\in J_n}\big\{A: |A_i\q-p_i|<\varphi_i(u), \ 1\le i\le d \big\}\supset \bigcup_{\alpha\in J_n}\prod_{i=1}^d\Delta\big(\R_{\alpha,i}, \psi_i(u_n)\big).
$$ Therefore $$
W(\varphi, \Phi)\supset \limsup_{n\to\infty}\bigcup_{\alpha\in J_n}\prod_{i=1}^d\Delta\big(\R_{\alpha,i}, \psi_i(u_n)\big).
$$
\end{itemize}

%
%
%
%

Now let $$
\widetilde{\rho}_i(u)={\varphi_i(u)}\cdot \left(\prod_{i=1}^d \varphi_i(u)\cdot \prod_{k=1}^h\Phi_k(u)\right)^{-1/d}, \ \ 1\le i\le d.
$$ 

\begin{lem}[Ubiquity for rectangles]\label{l5.3} With the notation above, the 
pair $(\{\R_{\alpha}\}_{\alpha\in J}, \beta)$ is a ubiquitous system with respect to the function $\rho$ and the sequences $\{\ell_n, u_n\}_{n\ge 1}$.
\end{lem}
\begin{proof} For any $u\in \N$,
by Minkowski's theorem, for any fixed matrix $A\in [0,1]^{dh}$ there exists a non-zero integer vector $(q_1,\dots, q_h, p_1,\dots, p_d)$ such that
\begin{align*}
\left\{
                                 \begin{array}{ll}
                                    |A_i\q-p_i|<\widetilde{\rho}_i(u), & \hbox{$1\le i\le d$;} \\
                                   |q_k|\le \Phi_k(u) & \hbox{$1\le k\le h$}.
                                 \end{array}
                               \right.
\end{align*} In other words, for any $u\in \N$ and $A\in [0,1]^{dh}$, there exists $\alpha\in J$ with $\beta_{\alpha}\le u$ such that\begin{align}\label{a1}
A\in \prod_{i=1}^d\Delta\left(\R_{\alpha, i}, \frac{\widetilde{\rho}_i(u)}{\max_{1\le k\le h}|q_k|}\right).
\end{align}

Recall $u_n=M^{n}$ and choose $\ell_n$ small enough such that \begin{align}\label{ln}
J_n&=\{\alpha: \ell_n\le \beta_{\alpha}\le u_n\}\nonumber\\
&\supset \Big\{(q_1,\dots, q_h, p_1,\dots, p_d): \frac{1}{M}\Phi_k(M^{n})\le |q_k|^+\le \Phi_k(M^{n}), 1\le k\le h\Big\}:=\widetilde{J}_n.
\end{align} Thus
\begin{align*}
  \Big\{\alpha\in J: \beta_{\alpha}\le u_n\Big\}&= \widetilde{J}_n\cup \bigcup_{j=1}^h\Big\{\alpha\in J: |q_j|\le \frac{\Phi_j(M^{n})}{M},\ |q_k|\le \Phi_k(M^{n}), \ {\text{for all}}\  k\ne j\Big\}\\
  :&=\widetilde{J}_n\cup \bigcup_{j=1}^hJ_{n,j}.
\end{align*}

Let $B=\prod_{i=1}^dB(x_i, r)$ be a ball in $[0,1]^{dh}$. Taking $u=u_n$ in (\ref{a1}), one has\begin{align*}
B&=B\cap \bigcup_{\alpha:\beta_{\alpha}\le u_n}\prod_{i=1}^d\Delta\left(\R_{\alpha, i}, \frac{\widetilde{\rho}_i(u_n)}{\max_{1\le k\le h}|q_k|}\right)\\
&=\left(B\cap \bigcup_{\alpha\in \widetilde{J}_n}\prod_{i=1}^d\Delta\left(\R_{\alpha, i}, \frac{\widetilde{\rho}_i(u_n)}{\max_{1\le k\le h}|q_k|}\right)\right)\bigcup \left(B\cap \bigcup_{j=1}^h\bigcup_{\alpha\in J_{n,j}}\prod_{i=1}^d\Delta\left(\R_{\alpha, i}, \frac{\widetilde{\rho}_i(u_n)}{\max_{1\le k\le h}|q_k|}\right)\right)\\
&=I_1\cup I_2.
\end{align*} We give an upper bound estimation on the measure of $I_2$:\begin{align*}
  \mathcal{L}(I_2)\le \sum_{j=1}^h\sum_{|q_j|\le \frac{\Phi_j(u_n)}{M}; |q_k|\le \Phi_k(u_n),  k\ne j}\ \sum_{p_1,\dots, p_d}\prod_{i=1}^d\mathcal{L}\left(B(x_i, r)\cap \Delta\left(\R_{\alpha, i}, \frac{\widetilde{\rho}_i(u_n)}{\max_{1\le k\le h}|q_k|}\right)\right)
\end{align*} For any fixed $(q_1,\dots, q_h)$, the number of $p_i$ such that the intersection is non-empty is at most $2r\cdot\max_{1\le k\le h}|q_k|+2$. Thus it follows that
\begin{align*}
  \mathcal{L}(I_2)&\le\sum_{j=1}^h\sum_{|q_j|\le \frac{\Phi_j(u_n)}{M}; |q_k|\le \Phi_k(u_n),  k\ne j}\ \Big(2r\cdot\max_{1\le k\le h}|q_k|+2\Big)^d\cdot \prod_{i=1}^d\frac{\widetilde{\rho}_i(u_n)\cdot r^{h-1}}{\max_{1\le k\le h}|q_k|}\\
  &=\sum_{j=1}^h\sum_{|q_j|\le \frac{\Phi_j(u_n)}{M}, |q_k|\le \Phi_k(u_n), \ k\ne j}\Big(2r\cdot\max_{1\le k\le h}|q_k|+2\Big)^d\cdot\left(\max_{1\le k\le h}|q_k|\right)^{-d}\cdot\left(\prod_{k=1}^h\Phi_k(u_n)\right)^{-1}\cdot r^{d(h-1)}.\end{align*}
 Then using a simple inequality that $(a+b)^d\le (2a)^d+(2b)^d$, it follows that\begin{align*}
\mathcal{L}(I_2)
&\le \frac{2^{2d}\cdot r^{dh}}{M}+\frac{h\cdot 2^{2d}\cdot r^{d(h-1)}\log \Phi_1(u_n)}{\Phi_1(u_n)}\le \frac{1}{2}\cdot \mathcal{L}(B)
\end{align*} whenever $M>2^{2d+3}$ and $n$ is large enough. Thus one gets \begin{align*}
  \mathcal{L}\left(B\cap \bigcup_{\alpha\in \widetilde{J}_n}\prod_{i=1}^d\Delta\left(\R_{\alpha, i}, \frac{\widetilde{\rho}_i(u_n)}{\max_{1\le k\le h}|q_k|}\right)\right)\ge \frac{1}{2}\cdot\mathcal{L}(B).
\end{align*}
Note also that for any $\alpha\in \widetilde{J}_n$, $$
\max_{1\le k\le h}|q_k|\ge \frac{1}{M}\cdot \max_{1\le k\le h}\Phi_k(u_k),
$$ which 
implies that \begin{align*}
  &\ \mathcal{L}\left(B\cap \bigcup_{\alpha\in {J}_n}\prod_{i=1}^d\Delta\left(\R_{\alpha, i}, \frac{M\cdot\widetilde{\rho}_i(u_n)}{\max_{1\le k\le h}\Phi_k(u_{n})}\right)\right)\ge \mathcal{L}\left(B\cap \bigcup_{\alpha\in \widetilde{J}_n}\prod_{i=1}^d\Delta\left(\R_{\alpha, i}, \frac{M\cdot\widetilde{\rho}_i(u_n)}{\max_{1\le k\le h}\Phi_k(u_{n})}\right)\right)\\
  \ge &\ \mathcal{L}\left(B\cap \bigcup_{\alpha\in \widetilde{J}_n}\prod_{i=1}^d\Delta\left(\R_{\alpha, i}, \frac{\widetilde{\rho}_i(u_n)}{\max_{1\le k\le h}|q_k|}\right)\right)\ge \frac{1}{2}\cdot\mathcal{L}(B).
\end{align*}
This shows the ubiquity property with the ubiquitous function $$
\rho_i(u)=\frac{M\cdot\widetilde{\rho}_i(u)}{\max_{1\le k\le h}\Phi_k(u)}, \ 1\le i\le d.
$$
\end{proof}

To summarize, we have \begin{itemize}
  \item the ubiquitous system $(\{\R_{\alpha}\}_{\alpha\in J}, \beta)$ with respect to $\rho$:  by Lemma \ref{l5.3}.

\item the $\lambda$-regularity property: by the monotonicity of $\varphi_i$ and the condition assumed on $\Phi_k$, namely$$
\psi_i(u_{n+1})=\frac{\varphi_i(u_{n+1})}{\max_{1\le k\le h}\Phi_k(u_{n+1})}\le \frac{1}{c_1}\cdot \frac{\varphi_i(u_n)}{\max_{1\le k\le h}\Phi_k(u_{n})}=\frac{1}{c_1}\cdot \psi_i(u_n).
$$
\end{itemize} Recall the definitions of $\rho_i$ and $\psi_i$: $$\rho_i(u)=\frac{M\cdot \varphi_i(u)}{\max_{1\le k\le h}\Phi_k(u)}\cdot \left(\prod_{i=1}^d\varphi_i(u)\cdot \prod_{k=1}^h\Phi_k(u)\right)^{-1/d},  \ \  \psi_i(u)=\frac{1}{h}\cdot \frac{\varphi_i(u)}{\max_{1\le k\le h}\Phi_k(u)}.$$
Then one has
\begin{itemize}
\item $\rho_i(u)\ge \psi_i(u)$: by (\ref{g2});

\item $\rho_i(u)\to 0$ as $u\to \infty$: by Lemma \ref{l3.1}, since the denominator in $\rho_i(u)$ tends to infinity.
\end{itemize}

Thus all the conditions in Theorem \ref{t1} are satisfied, and then we apply it to arrive at the desired result, i.e. $$
\mathcal{L}\big(W(\varphi, \Phi)\big)=1 \ \Longleftarrow\ \sum_{t=0}^{\infty}\prod_{i=1}^d\varphi_i(M^t)\prod_{k=1}^h\Phi_k(M^{t})=\infty.
$$
Finally, by monotonicity of $\varphi$ and the assumption on $\Phi$, one has \begin{align*}
  \infty=\sum_{q=1}^{\infty}q^{-1}\cdot \prod_{i=1}^d\varphi_i(q)\cdot \prod_{k=1}^h\Phi_k(q)&=\sum_{t=0}^{\infty}\sum_{M^t\le q<M^{t+1}}q^{-1}\cdot \prod_{i=1}^d\varphi_i(q)\cdot \prod_{k=1}^h\Phi_k(q)\\
  &\le (M-1)c_2^h\cdot \sum_{t=0}^{\infty}\prod_{i=1}^d\varphi_i(M^t)\prod_{k=1}^h\Phi_k(M^{t}).
\end{align*} This proves the divergence part of Theorem \ref{schmidt}, while the convergence part follows easily from the Borel-Cantelli lemma.
\end{proof}


\begin{thebibliography}{99}


\bibitem{AllBa}  D. Allen and S. Baker, {\em A general mass transference principle}, Selecta Math. (N.S.) \textbf{25} (2019), no. 3, Art. 39, 38 pp.

%
%

\bibitem{BadBV} D. Badziahin, V. Beresnevich and S. Velani, {\em Inhomogeneous theory of dual Diophantine approximation on manifolds}, Adv. Math. {\bf 232} (2013), 1--35.

\bibitem{BakS} A. Baker and W. Schmidt, {\em Diophantine approximation and Hausdorff dimension,} Proc. London Math.
Soc. (3) \textbf{21} (1970), 1--11.

\bibitem{Ber99} V. Beresnevich, {\em On approximation of real numbers by real algebraic numbers}, Acta Arith. \textbf{90} (1999),  no. 2, 97--112.

\bibitem{Ber12}  \bysame, {\em Rational points near manifolds and metric Diophantine approximation}, Ann. of Math. (2) {\bf 175 }(2012), no. 1, 187--235.


\bibitem{BBD02} V. Beresnevich, V. Bernik and M. Dodson,
{\em Regular systems, ubiquity and Diophantine approximation. A panorama of number theory or the view from Baker's garden} (Z\"{u}rich, 1999), 260--279, Cambridge Univ. Press, Cambridge (2002).

\bibitem{BDV06} V. Beresnevich, D. Dickinson  and S. Velani,  {\em Measure theoretic laws for lim sup sets}, Mem. Amer.
Math. Soc. \textbf{179} (2006), no. 846, x+91.

\bibitem{B2}  \bysame, {\em Diophantine approximation on planar curves and the distribution of rational points. With an Appendix II by R. C. Vaughan}, Ann. of Math. (2) \textbf{166}, (2007), no. 2, 367--426


\bibitem{BerD} V. Bernik and M. Dodson, Metric Diophantine approximation on manifolds, Cambridge Tracts in Mathematics,
vol. 137. Cambridge University Press, Cambridge (1999).

\bibitem{Bug} Y. Bugeaud, {\em Approximation by algebraic integers and Hausdorff dimension},
J. London Math. Soc. (2) \textbf{65} (2002), no. 3, 547--559.

\bibitem{Bug03} \bysame, {\em A note on inhomogeneous Diophantine approximation},
Glasg. Math. J. 45 (2003), no. 1, 105--110.

\bibitem{BoC} M. Boshernitzan and J. Chaika, {\em Borel-Cantelli sequences}, J. Anal. Math. {\bf 117} (2012), 321--345.

\bibitem{B1} V. Bernik, D. Kleinbock and G. Margulis, {\em Khintchine-type theorems on manifolds: the convergence case for standard and multiplicative versions}, Int. Math. Res. Not. (2001), no. 9, 453--486.


\bibitem{Chung} K. L. Chung and P. Erd\"{o}s, {\em On the application of the Borel-Cantelli lemma}, Trans. Amer. Math. Soc. \textbf{72} (1952), 179--186.

\bibitem{DickV}  D. Dickinson and S. Velani, {\em Hausdorff measure and linear forms}, J. Reine Angew. Math. {\bf 490} (1997), 1--36.

\bibitem{Dick99} D. Dickinson, M. Dodson and J. Yuan, {\em Hausdorff dimension and p-adic Diophantine approximation}, Indag. Math. (N.S.) {\bf 10} (1999), no. 3, 337--347.

\bibitem{Dod92} M. Dodson, {\em Hausdorff dimension, lower order and Khintchine's theorem in metric Diophantine approximation}, J. Reine Angew. Math. {\bf 432} (1992), 69--76.

\bibitem{Dodson} \bysame, {\em Geometric and probabilistic ideas in the metric theory of Diophantine approximations} (Russian), Uspekhi Mat. Nauk {\bf 48} (1993), no. 5(293), 77--106; translation in Russian Math. Surveys {\bf 48} (1993), no. 5, 73--102.

\bibitem{DoRV} M. Dodson, B. Rynne and J. Vickers, {\em Diophantine approximation and a lower bound for Hausdorff
dimension}, Mathematika \textbf{37} (1990), no. 1, 59--73.



\bibitem{Gall} P. Gallagher, {\em Metric simultaneous diophantine approximation}, J. London Math. Soc. \textbf{37} (1962), 387--390.


\bibitem{Gro} A. Groshev, {\em Un Th\'{e}or\`{e}me sur les syst\`{e}mes des formes lin\'{e}aires}, Dokl. Akad. Nauk SSSR \textbf{19} (1938), 151--152.

\bibitem{Hein} J. Heinonen, Lectures on Analysis on Metric Spaces, Springer, New York (2001).

\bibitem{HusY} M. Hussain and T. Yusupova, {\em A note on the weighted Khintchine-Groshev theorem}, J. Th\'{e}or. Nombres Bordeaux {\bf 26} (2014), no. 2, 385--397.

\bibitem{Hutch} J. Hutchinson,  {\em Fractals and self-similarity}, Indiana Univ. Math. J. \textbf{30} (1981), no. 5, 713--747.


\bibitem{Jarn1} V. Jarnik, {\em Diophantischen Approximationen und Hausdorffsches Mass}, Mat. Sb. \textbf{36}  (1929), 371--381.

\bibitem{Khint} A. Khintchine, {\em Einige S\"{a}tze \"{u}ber Kettenbr\"{u}che, mit Anwendungen auf die Theorie der Diophantischen Approximationen}, Math. Ann. \textbf{92} (1924), 115--125.

\bibitem{KoS}  S.  Kochen and C. Stone, \emph{A note on the Borel-Cantelli lemma}, Illinois J. Math. \textbf{8} (1964), 248--251.

\bibitem{Kris} S. Kristensen, {\em On well-approximable matrices over a field of formal series}, Math. Proc. Cambridge Philos. Soc. {\bf 135} (2003), no. 2, 255--268.


\bibitem{Lev} J. Levesley, {\em
A general inhomogeneous Jarnik-Besicovitch theorem}, J. Number Theory \textbf{71} (1998), no. 1, 65--80.

\bibitem{Matt} P. Mattila, Geometry of Sets and Measures in Euclidean Spaces, Cambridge Studies in Advanced
Mathematics, vol. 44. Cambridge University Press, Cambridge (1995).



\bibitem{Ryn92} B. Rynne, {\em Regular and ubiquitous systems, and $M_{\infty}$-dense sequences}, Mathematika {\bf 39} (1992), no. 2, 234--243.


\bibitem{Mi} H. Minkowski, {\em Geometrie der Zahlen}, Teubner< Leipzig, Berlin (1986).


\bibitem{Sch1} W. M. Schmidt, {\em A metrical theorem in diophantine approximation.} Canadian J. Math. \textbf{12}, (1960), 619--631.

\bibitem{Sch} \bysame, Diophantine approximation, Lecture Notes in Mathematics, vol. 785,  Springer, Berlin x+299 pp (1980).

\bibitem{Sprindzuk}
V.~G. Sprind\v{z}uk, Metric theory of {D}iophantine approximations,  John Wiley, 1979, Translated by R.\,A.\ Silverman.


\bibitem{WWX} B. Wang, J. Wu and J. Xu, {\em Mass transference principle for lim sup sets generated by rectangles}, Math. Proc. Cambridge Philos. Soc.
\textbf{158} (2015), 419--437.

\bibitem{WangW} B. Wang and J. Wu, {\em Mass transference principle from rectangles to rectangles in Diophantine approximation}, Math. Ann. {\bf 381} (2021), no. 1--2, 243--317.













\end{thebibliography}
\end{document}